\documentclass[11pt]{article}
\usepackage{amsfonts, amsmath, amssymb, latexsym, eucal, amscd} 
\usepackage{cite}
\usepackage[dvips]{pict2e}
\usepackage[all]{xy}
\newtheorem{theorem}{Theorem}[section]
\newtheorem{lemma}[theorem]{Lemma}

\newtheorem{corollary}[theorem]{Corollary}
\newtheorem{exAux}[theorem]{Example}

\newtheorem{Def}[theorem]{Definition}
\newenvironment{defi}{\begin{Def} \rm}{\end{Def}}
\newtheorem{Note}[theorem]{Note}
\newenvironment{note}{\begin{Note} \rm}{\end{Note}} 
\newtheorem{Problem}[theorem]{Problem}

\newtheorem{Rem}[theorem]{Remark}

\newtheorem{Not}[theorem]{Notation}

\newtheorem{Conj}[theorem]{Conjecture}

\newtheorem{Ass}[theorem]{Assumption}

\newenvironment{proof}{\medskip\noindent{\bf Proof.\ }}{\qed\medskip}

\newcommand{\qed}{\hfill\mbox{$\Box$\qquad\qquad}}

\newcommand{\C}{\mathbb{C}}
\newcommand{\R}{\mathbb{R}}

\newcommand{\Mat}{\text{\rm Mat}}

\renewcommand{\th}{\theta}

\newcommand{\N}{{\mathcal N}}
\renewcommand{\sl}{\mathfrak{sl}_2}
 
\newif\ifDRAFT

\begin{document}

\title{The nucleus of the Johnson graph $J(N,D)$}

\author{Kazumasa Nomura and Paul Terwilliger}

\maketitle

{\bf Abstract}.
This paper is about the nucleus of the Johnson graph $\Gamma = J(N,D)$
 with  $N > 2D$.
The nucleus is described as follows.
Let $X$ denote the vertex set of $\Gamma$.
Let $A \in \Mat_X(\C)$ denote the adjacency matrix of $\Gamma$.
Let $\{E_i\}_{i=0}^D$ denote the $Q$-polynomial ordering of the primitive idempotents of $A$.
Fix $x \in X$.
The corresponding dual adjacency matrix
$A^*$ is the diagonal matrix  in $\Mat_X(\C)$ such that for $y \in X$ the $(y,y)$-entry of $A^*$
is equal to the $(x,y)$-entry of $|X| E_1$.
For $0 \leq i \leq D$ the diagonal matrix $E^*_i \in \Mat_X(\C)$ is the projection
onto the $i$th subconstituent of $\Gamma$ with respect to $x$.
The matrices $\{E^*_i\}_{i=0}^D$ are the primitive idempotents of $A^*$.
The subalgebra $T$ of $\Mat_X(\C)$ generated by $A$, $A^*$ is called the
subconstituent algebra of $\Gamma$ with respect to $x$.
Let $V=\C^X$ denote the standard module of $\Gamma$.
For $0 \leq i \leq D$ define
\[
 \N_i = (E^*_0 V + E^*_1 V + \cdots + E^*_i V) \cap
          (E_0 V + E_1 V + \cdots + E_{D-i} V).
\]
It is known  that the sum $\N = \sum_{i=0}^D \N_i$ is direct,  and
$\N$ is a $T$-module.
This $T$-module $\N$ is  called the nucleus of $\Gamma$ with respect to $x$.
The sum $\N = \sum_{i=0}^D E^*_i \N$  is direct.
In this paper we consider the following vectors in $V$.
For a subset $\alpha \subseteq x$ let $\alpha^\vee \in V$ denote the characteristic vector
of the set of vertices in $X$ that contain $\alpha$.
Let $\alpha^\N  \in V$ denote the characteristic vector
of the set of vertices in $X$ whose intersection with $x$ is equal to $\alpha$.
We show that
(i)
the vectors $\{\alpha^\vee\}_{\alpha \subseteq x}$ form a basis of $\N$;
(ii)
the vectors $\{\alpha^\N\} _{\alpha \subseteq x}$ form a basis of $\N$;
 (iii)
for $0 \leq i \leq D$,
the vectors $\{ \alpha^\N \;|\; \alpha \subseteq x, \; |\alpha| = D-i\}$
 form a basis of $E^*_i \N$;
 (iv)
for $0 \leq i \leq D$,
the vectors $\{ \alpha^\vee \;|\; \alpha \subseteq x, \; |\alpha| = D-i\}$
 form a basis of $\N_i$.
We give the transition matrices between the bases  
$\{\alpha^\vee\}_{\alpha \subseteq x}$ and 
$\{\alpha^\N\}_{\alpha \subseteq x}$.
We  give the actions of $A$, $A^*$ on the bases 
$\{\alpha^\vee\}_{\alpha \subseteq x}$ and
$\{\alpha^\N\}_{\alpha \subseteq x}$.
For $0 \leq i \leq D$ we characterize the basis vectors for $E^*_i \N$
in terms of the connected components of a certain graph 
obtained by adjusting the edges of  the subgraph of $\Gamma$ induced on
the $i$th subconstituent of $\Gamma$ with respect to $x$.
\bigskip

\section{Introduction}
\label{sec:intro}
\ifDRAFT {\rm sec:intro}. \fi

In \cite{T:nuc} the second author introduced the nucleus of a $Q$-polynomial
distance-regular graph.
In this paper we describe the nucleus of the Johnson graph $J(N,D)$.

Before we summarize our results, we describe the context and history.
Let $\Gamma$ denote  a distance-regular graph with vertex set $X$ and diameter $D$,
see  \cite[Section 4.1.A]{BCN}.
Let $\partial$ denote the path-length distance function.
Let $A \in \Mat_X(\C)$  denote the adjacency matrix  of $\Gamma$,
see \cite[p.\ $2$]{BBIT}.
Let  $\{E_i\}_{i=0}^D$ denote a $Q$-polynomial ordering of the primitive idempotents
 of $A$,
see  \cite[Section 2]{T:notes},    \cite[Definition 2.80]{BBIT}.
Fix a vertex $x \in X$.
For $0 \leq i \leq D$, the set $\Gamma_i(x) = \{ y \in X \;|\; \partial(x,y) = i\}$ is called the
$i$th subconstituent of $\Gamma$ with respect to $x$.
The corresponding dual primitive idempotent is the diagonal matrix $E^*_i \in \Mat_X(\C)$
such that for  $y \in X$ the $(y,y)$-entry of $E^*_i$ is $1$ if $y \in \Gamma_i(x)$ and $0$ if
 $y \not\in \Gamma_i(x)$.
Define a diagonal matrix $A^* \in \Mat_X(\C)$ such that 
for $y \in X$ the $(y,y)$-entry of $A^*$ is equal to the $(x,y)$-entry of $|X| E_1$.
The matrices $\{E^*_i\}_{i=0}^D$ are the primitive idempotents of $A^*$, see \cite[Section 3]{T:subconst1}.
The subalgebra $T$ of $\Mat_X(\C)$ generated by $A$, $A^*$
is called the subconstituent algebra of $\Gamma$ with respect to $x$, see \cite[Section 3]{T:subconst1}.
Let $V = \C^X$ denote the vector space over $\C$ consisting of the column vectors with coordinates
indexed by $X$ and all entires in $\C$.
For $y \in X$ the vector $\widehat{y} \in V$ has $y$-coordinate $1$ and all other coordinates $0$.
The vectors $\{\widehat{y}\}_{y \in X}$ form a basis of $V$.
For $0 \leq i  \leq D$ define a subspace $\N_i$ of $V$ by
\[
 \N_i =
  (E^*_0 V + E^*_1 V + \cdots + E^*_i V)  \cap
  (E_0 V + E_1 V + \cdots + E_{D-i}V).
\]
By \cite[Corollay 5.8]{T:disp} the sum $\N = \sum_{i=0}^D \N_i$ is direct.
By \cite[Theorem 7.1]{T:disp} $\N$ is a $T$-module.
This $T$-module $\N$ is the nucleus of $\Gamma$ with respect to $x$.
In \cite{T:nuc} the nucleus is described for   the hypercubes, the Odd graphs, the Hamming graphs,
and the nonbipartite dual polar graphs.

In this paper we consider the Johnson graph $J(N,D)$.
This graph is defined as follows.
Fix nonnegative integers $N$, $D$ such that $2D < N$.
Let $\Omega$ denote a set of size $N$.
Let the set $X$ consist of the
subsets of $\Omega$ that have size $D$.
The graph $J(N, D)$ has vertex set $X$. Vertices $x,y \in X$ are
adjacent whenever $|x \cap y | = D-1$.
By \cite[Theorem 9.1.2]{BCN} $J(N, D)$ is distance-regular with diameter $D$.
The adjacency matrix 
$A$ has eigenvalues
\[
\th_i = (D-i)(N-D-i) - i \qquad\qquad (0 \leq i \leq D).
\]
For $0 \leq i \leq D$ let $E_i$ denote the primitive idempotent of $A$ associated with $\th_i$.
The ordering $\{E_i\}_{i=0}^D$ 
is $Q$-polynomial, see \cite[Corollary to Theorem 2.6]{BI}.
Recall the fixed vertex $x \in X$.
We consider the following vectors in $V$.
For a subset $\alpha \subseteq x$, define
\[
\alpha^\vee =
\sum_{\text{\scriptsize $\begin{matrix} y \in X \\ \alpha \subseteq y \end{matrix}$}} \widehat{y},
\qquad\qquad\qquad
\alpha^\N =
\sum_{\text{\scriptsize $\begin{matrix} y \in X \\ y \cap x = \alpha \end{matrix}$}} \widehat{y}.
\]
The vector $\alpha^\vee$ is the characteristic vector
of the set of vertices in $X$ that contain $\alpha$,
and  $\alpha^\N$ is  the characteristic vector
of the set of vertices in $X$ whose intersection with $x$ is equal to $\alpha$.
We show the following:
\begin{itemize}
\item[(i)]
the vectors $\{\alpha^\vee\}_{\alpha \subseteq x}$ form a basis of $\N$;
\item[(ii)]
the vectors $\{\alpha^\N\}_{\alpha \subseteq x}$ form a basis of $\N$;
\item[\rm (iii)]
for $0 \leq i \leq D$,
the vectors $\{ \alpha^\vee \;|\; \alpha \subseteq x, \; |\alpha| = D-i\}$
 form a basis of $\N_i$;
\item[\rm (iv)]
for $0 \leq i \leq D$,
the vectors $\{ \alpha^\N \;|\; \alpha \subseteq x, \; |\alpha| = D-i\}$
 form a basis of $E^*_i \N$.
\end{itemize}
We give the transition matrices between the bases 
 $\{\alpha^\vee\}_{\alpha \subseteq x}$ and
$\{\alpha^\N\}_{\alpha \subseteq x}$.
We  give the actions of $A$, $A^*$ on the bases 
$\{\alpha^\vee\}_{\alpha \subseteq x}$ and
$\{\alpha^\N\}_{\alpha \subseteq x}$.
We  characterize the vectors $\{\alpha^\N\}_{\alpha \subseteq x}$
as follows.
For $0 \leq i \leq D$ we show that the subgraph $\Gamma_i(x)$ is isomorphic to the Cartesian product
 $J(D,i) \times J(N-D,i)$.
So $\Gamma_i(x)$ has two types of edges, some come from $J(D,i)$ and some come from
$J(N-D,i)$.
Consider a new graph obtained from $\Gamma_i(x)$ by removing the edges that come from $J(D,i)$.
The vectors  $\{ \alpha^\N \;|\; \alpha \subseteq x, \; |\alpha| = D-i\}$ are the characteristic
vectors of the connected components of this new graph.

This paper is organized as follows.
 Section \ref{sec:pre} contains some preliminaries.
In Section \ref{sec:DRG} we review $Q$-polynomial distance-regular graphs
and the subconstituent algebra $T$.
In Section \ref{sec:Tmodule}
we review $T$-modules.
In Section \ref{sec:nuc} we recall the notion of the nucleus.
In Sections \ref{sec:Johnson}--\ref{sec:main} we consider the Johnson graph.
In Section \ref{sec:Johnson}  we review the definition and basic properties of the Johnson graph $J(N,D)$.
In Section \ref{sec:Gammai} we investigate the structure of the subgraph $\Gamma_i(x)$.
In Section \ref{sec:alphaN}
we introduce the vectors $\alpha^\N$ and $\alpha^\vee$, and give relations
between these vectors.
In Section \ref{sec:action}
we describe the action of $A$, $A^*$ on the vectors  $\alpha^\N$ and $\alpha^\vee$.
In Section \ref{sec:alphavee}
we prove some results concerning $\alpha^\vee$.
In Section \ref{sec:main} 
we use the vectors $\alpha^\N$ and $\alpha^\vee$ to obtain two bases for the nucleus.

\section{Preliminaries}
\label{sec:pre}
\ifDRAFT {\rm sec:pre}. \fi

We now begin our formal argument.
The following notation and concepts will be used throughout the paper.
Let $\C$  denote the  field of complex numbers.
Let $X$ denote a finite nonempty set.
Let $\Mat_X(\C)$ denote the $\C$-algebra consisting of the matrices with
rows and columns  indexed by $X$ and all entries in $\C$.
The identity matrix in $\Mat_X(\C)$ is denoted by $I$,
and let $J \in \Mat_X(\C)$ have all entries $1$.
Let $\C^X$ denote the vector space over $\C$ consisting of the column vectors
with  coordinates indexed by $X$ and all entries  in $\C$.
The algebra $\Mat_X(\C)$ acts on $\C^X$ by left multiplication.
We call $\C^X$ the {\em standard module}.
For a subset $Y \subseteq X$, the vector
$\widehat{Y} = \sum_{y \in Y} \widehat{y}$ is called the {\em characteristic vector of $Y$}.

All the graphs we discuss are assumed to be finite, undirected, and have no loops and no multiple edges.
Let $\Gamma$ denote a graph with vertex set $X$.
Then $\Gamma$ is said to be {\em regular with valency $k$} whenever each vertex  of $\Gamma$ is adjacent to
exactly $k$ vertices.
Define $A \in \Mat_X(\C)$ that has $(x,y)$-entry
\[
  A_{x,y} = 
 \begin{cases}
   1 & \text{\rm if  $x$, $y$ are adjacent},
\\
  0 & \text{\rm if  $x$, $y$ are not adjacent}
\end{cases}
\qquad\qquad
(x,y \in X).
\]
We call $A$ the {\em adjacency matrix of $\Gamma$}.
We recall  the Cartesian product of graphs.  
Given graphs $\Gamma$ and $\Gamma'$ with vertex set $X$ and $X'$,
the Cartesian product  $\Gamma \times \Gamma'$ is the graph with vertex set $X \times X'$
and adjacency defined as follows.
Two vertices $(x, x')$, $ (y, y')$ are adjacent in $\Gamma \times \Gamma'$ whenever 
either 
(i) $x$, $y$ are adjacent in $\Gamma$ and  $x' = y'$
or
(ii) $x = y$ and $x'$, $y'$ are adjacent in $\Gamma'$.
In case (i) (resp.\ (ii)) we say that  the edge $(x, x'), (y, y')$ in $\Gamma \times \Gamma'$
{\em comes from  $\Gamma$} (resp.\ {\em comes from $\Gamma'$}).
We make an observation.

\begin{lemma}   \label{lem:cartesian}   \samepage
\ifDRAFT {\rm lem:cartesian}. \fi
Let $\Gamma$ and $\Gamma'$ denote connected graphs with vertex set $X$ and $X'$,
respectively.
Consider a new graph obtained from $\Gamma \times \Gamma'$ by removing the
edges that come from $\Gamma$.
Then the connected components of the new graph are
\[
  \{x\} \times X', \qquad\qquad  x \in X.
\]
In particular, every connected component of the new graph is isomorphic to $\Gamma'$,
and the number of connected components is $|X|$.
\end{lemma}

\begin{defi}    \samepage
For an integer $n \geq 0$ and a set $Y$ of size at least $n$, let the
set $\binom{Y}{n}$ consist of the subsets of $Y$ that  have size $n$.
Thus
\[
  \binom{Y}{n} = \{y \;|\; y \subseteq Y, \; |y| = n\}.
\]
\end{defi}

\begin{defi}   \samepage
For sets $Y, Z$, let the  set $Y \setminus Z$ consist of
the elements in $Y$ that are not in $Z$.
Thus
\[
Y \setminus Z = \{ y \in Y \;|\; y \not\in Z\}.
\]
\end{defi}

\section{$Q$-polynomial  distance-regular graphs}
\label{sec:DRG}
\ifDRAFT {\rm sec:DRG}. \fi

In this section we recall $Q$-polynomial distance-regular graphs and the
subconstituent algebra.
Let $\Gamma$ denote a connected graph with vertex set $X$.
For $x,y \in X$, let $\partial(x,y)$ denote the path-length distance between $x$, $y$.
The integer
\[
D = \text{max} \{\partial (x,y) \;|\; x,y \in X\}
\]
is called the {\em diameter of $\Gamma$}.
For $x \in X$ and $0 \leq i \leq D$ define
\[
 \Gamma_i (x) =\{ y \in X \;|\; \partial(x,y) = i \}.
\]
The graph $\Gamma$ is said to be {\em distance-regular}
whenever for $0 \leq h,i,j \leq D$ and $x,y \in X$ at distance $\partial (x,y) = h$,
the number
\[
 p^h_{i,j}  = |\Gamma_i (x) \cap \Gamma_j(y)|
\]
is a constant that is independent of $x, y$.
The parameters $p^h_{i,j}$ are called the {\em intersection numbers} of $\Gamma$.

For the rest of this paper, we assume that $\Gamma$ is distance-regular.
Define
\[
  k_i = p^0_{i,i}  \qquad \qquad (0 \leq i \leq D).
\]
For $x \in X$ we have $|\Gamma_i(x) | = k_i$ for $0 \leq i \leq D$.
The graph $\Gamma$ is regular,
with  valency $k=k_1$ if $D \geq 1$ and  $k=0$ if $D=0$.
By the triangle inequality, the following hold for $0 \leq h,i,j  \leq D$:
\begin{itemize}
\item[\rm (i)]
$p^h_{i,j} = 0$ if one of $h,i,j$ is greater than the sum of the other two;
\item[\rm (ii)]
$p^h_{i,j} \neq  0$ if one of $h,i,j$ is equal to the sum of the other two.
\end{itemize}
By \cite[Lammma 4.1.7]{BCN} all the intersection numbers are determined by
the following intersection numbers:
\begin{align*}
c_i &= p^i_{1, i-1} \; (1 \leq i \leq D),
&
a_i &= p^i_{1,i} \; (0 \leq i \leq D),
&
b_i &= p^i_{1,i+1} \; (0 \leq i \leq D-1).
\end{align*}
We have
\[
c_i + a_i + b_i = k    \qquad\qquad (0 \leq i \leq D),
\]
where $c_0 = 0$ and $b_D = 0$.
We recall the Bose-Mesner algebra of $\Gamma$.
For $0 \leq i \leq D$, define $A_i \in \Mat_X(\C)$ that has $(x,y)$-entry
\[
 (A_i)_{x,y} =
   \begin{cases}
    1  &  \text{ if $\partial(x,y) = i$ }, \\
    0  &   \text{ if $\partial(x,y) \neq i$ }
  \end{cases}
  \qquad\qquad  (x,y \in X).
\]
We have
(i) $A_0 = I$;
(ii) $J = \sum_{i=0}^D A_i$;
(iii) $A^t_i = A_i$ $(0 \leq i \leq D)$;
(iv) $\overline{A_i} = A_i$  $(0 \leq i \leq D)$;
(v) $A_i A_j = A_j A_i = \sum_{h=0}^D p^h_{i,j} A_h$ $(0 \leq i,j \leq D)$.
Therefore, the matrices $\{A_i\}_{i=0}^D$ form a basis for a
commutative subalgebra $M$ of $\Mat_X(\C)$,
called the {\em Bose-Mesner algebra of $\Gamma$}.
The adjacency matrix $A$ of $\Gamma$ satisfies  $A=A_1$ if $D \geq 1$ and $A=0$ if $D=0$.
By \cite[Corollary 3.4]{T:monster} $A$ generates $M$.
Since $A$ is symmetric with real entries, $A$ is diagonalizable over $\R$.
Consequently  $M$ has a second basis $\{E_i\}_{i=0}^D$ such that
(i) $E_0  = |X|^{-1} J$;
(ii) $I = \sum_{i=0}^D E_i$;
(iii) $E^t_i = E_i \;\; ( 0 \leq i \leq D)$;
(iv) $\overline{E_i} = E_i \;\; ( 0 \leq i \leq D)$;
(v) $E_i E_j = \delta_{i,j} E_i \;\; ( 0 \leq i,j \leq D)$.
We call $\{E_i\}_{i=0}^D$ the {\em primitive idempotents} of $\Gamma$.
We abbreviate  $V = \C^X$.
By construction,
\begin{equation}
  V = \sum_{i=0}^D E_i V  \qquad\qquad (\text{direct sum}).   \label{eq:sumEiV}
\end{equation}
The summands in \eqref{eq:sumEiV} are the eigenspaces of $A$.  
For $  0 \leq i \leq D$ let $\th_i$ denote the eigenvalue of $A$ for $E_i V$.
We have
$A E_i = \th_i E_i =  E_i A$ $(0 \leq i \leq D)$
and $A = \sum_{i=0}^D \th_i E_i$.
By \cite[p 128]{BCN} we have $\th_0 = k$.
We define some polynomials in an indeterminate $\lambda$.
For $0 \leq i \leq D$ define
\begin{align*}
\tau_i(\lambda) 
&= (\lambda - \th_0) (\lambda - \th_1) \cdots ( \lambda - \th_{i-1}),
\\
\eta_i (\lambda) &=
(\lambda - \th_D)(\lambda - \th_{D-1}) \cdots (\lambda - \th_{D-i+1}).
\end{align*}
For $0 \leq r \leq D$ we have
\begin{equation}
E_r =
\prod_{\text{\scriptsize $\begin{matrix} 0 \leq i \leq D \\ i \neq r \end{matrix}$ }}
        \frac{A- \th_i I}{\th_r- \th_i}
=
\frac{ \tau_r (A) \eta_{D-r}(A) } { \tau_r(\th_r) \eta_{D-r}(\th_r) }.
   \label{eq:Er}
\end{equation}

Next we recall the Krein parameters of $\Gamma$.
We have $A_i \circ A_j = \delta_{i,j} A_i \;\; (0 \leq i,j \leq D)$,
where $\circ$ denotes the entry-wise multiplication.
Therefore, $M$ is closed under $\circ$.
Consequently,
there exist scalars $q^h_{i,j} \in \C$ $(0 \leq h,i,j \leq D)$ such that
\[
E_i \circ E_j = |X|^{-1} \sum_{h=0}^D q^h_{i,j} E_h
 \qquad \qquad (0 \leq i,j \leq D).
\]
The scalars $q^h_{i,j}$ are called the {\em Krein parameters} of $\Gamma$.
By \cite[p.\ 69]{BI} the Krein parameters  are non-negative real numbers.

Next we recall the $Q$-polynomial property.
The graph $\Gamma$ is said to be {\em $Q$-polynomial} (with respect to the ordering $\{E_i\}_{i=0}^D$)
whenever the following hold for $0 \leq h,i,j \leq D$:
\begin{itemize}
\item[\rm (i)]
$q^h_{i,j} = 0$ if one of $h,i,j$ is greater than the sum of the other two;
\item[\rm (ii)]
$q^h_{i,j} \neq 0$ if one of $h,i,j$ is equal the sum of the other two.
\end{itemize}
For the rest of this paper, we assume that $\Gamma$ is $Q$-polynomial 
with respect to the ordering $\{E_i\}_{i=0}^D$.

Next we recall the dual Bose-Mesner algebra of $\Gamma$.
Fix a vertex $x \in X$.
For $0 \leq i \leq D$, define a diagonal matrix $E^*_i = E^*_i (x) \in \Mat_X(\C)$ that 
has $(y,y)$-entry
\[
  (E^*_i)_{y,y} =
   \begin{cases}
    1  &  \text{ if $\partial(x,y) = i$}, \\
    0  &   \text{ if $\partial(x,y) \neq i$ }
  \end{cases}
  \qquad\qquad  ( y \in X).
\]
We have
(i) $I = \sum_{i=0}^D E^*_i$;
(ii) $E^*_i E^*_j = \delta_{i,j}E^*_i \;\; (0 \leq i,j \leq D)$.
Therefore the matrices $\{E^*_i\}_{i=0}^D$ form a basis for a commutative subalgebra  $M^* = M^*(x)$
of $\Mat_X(\C)$.
We call  $M^*$ the {\em dual Bose-Mesner algebra of $\Gamma$ with respect to $x$}.
We call $\{E^*_i\}_{i=0}^D$ the {\em dual primitive idempotents of $\Gamma$ with respect to $x$}.
We have
\begin{equation}
 V = \sum_{i=0}^D E^*_i V  \qquad\qquad (\text{direct sum}).    \label{eq:Vdecomp2}
\end{equation}
Define a matrix $A^* = A^*(x) \in \Mat_X(\C)$ as follows.
If $D=0$ then $A^*=0$.
If $D \geq 1$ then
 $A^*$ is diagonal with $(y,y)$-entry
\[
   (A^*)_{y,y} = |X| (E_1)_{x,y}    \qquad\qquad (y \in X).
\]
The matrix $A^*$ is called the {\em dual adjacency matrix of $\Gamma$ with respect to $x$}.
By \cite[p.\ 279]{T:subconst1},
there exist scalars $\th^*_i \in \C$ $(0 \leq i \leq D)$ such that
\[
  A^* = \sum_{i=0}^D \th^*_i E^*_i.
\]
By \cite[Lemma 3.11]{T:subconst1},  $\{\th^*_i\}_{i=0}^D$ are mutually distinct.
For $0 \leq i \leq D$, $\th^*_i$ is the eigenvalue of $A^*$ for the eigenspace  $E^*_i V$,
and
$E^*_i A^* = A^* E^*_i = \th^*_i E^*_i$.
The scalars $\{\th^*_i\}_{i=0}^D$ are called the {\em dual eigenvalues} of $\Gamma$.
By \cite[Corollary 11.6]{T:monster},
$A^*$ generates $M^*$. 
Let $T=T(x)$ denote the subalgebra of  $\Mat_X(\C)$ generated by $A$ and $A^*$.
The algebra $T$ is called the {\em subconstituent algebra of $\Gamma$
with respect to $x$}.

\section{$T$-modules}
\label{sec:Tmodule}
\ifDRAFT {\rm sec:Tmodule}. \fi

We continue to discuss the $Q$-polynomial distance-regular graph $\Gamma$
with vertex set $X$.
Fix $x \in X$ and write $T=T(x)$.
Recall the standard module $V = \C^X$.
In this section we discuss the $T$-modules.

By a {\em $T$-module} we mean a subspace $W$ of $V$ such that $T W \subseteq W$.
We recall the notion of isomorphism for $T$-modules.
Let $W$, $W'$ denote $T$-modules.
By a {\em $T$-module isomorphism from $W$ to $W'$}
we mean a $\C$-linear bijection $\sigma : W \to W'$
such that  $\sigma B = B \sigma$ for all $B \in T$.
The $T$-modules $W$ and $W'$ are said to be {\em isomorphic}
whenever there exists a $T$-module isomorphism from $W$ to $W'$.

A $T$-module $W$ is said to be {\em irreducible} whenever
$W \neq 0$ and $W$ does not contain a $T$-submodule besides $0$ and $W$.      
Let $W$ denote an irreducible $T$-module.
By the {\em endpoint of $W$} we mean
\[
   \text{\rm min} \{ i \;|\; 0 \leq i \leq D, \; E^*_i W \neq 0 \}.
\]
By the {\em dual endpoint of $W$} we mean
\[
   \text{\rm min} \{ i \;|\; 0 \leq i \leq D, \; E_i W \neq 0 \}.
\]
By \cite[Lemma 5.6]{T:nuc},
\begin{equation}
  | \{ i \;|\, 0 \leq i \leq D, \;  E_i W \neq 0\} |
=   | \{ i \;|\, 0 \leq i \leq D, \;  E^*_i W \neq 0\} |.        \label{eq:diam}
\end{equation}
By the {\em diameter of $W$} we mean the common value in \eqref{eq:diam} minus $1$.

\begin{lemma}   {\rm (See \cite[Lemma 5.4]{T:nuc}.) }
\samepage
Let $W$ denote an irreducible $T$-module,
with endpoint $r$, dual endpoint $t$, and diameter $d$.
Then for $0 \leq i \leq D$ the following hold:
\begin{itemize}
\item[\rm (i)]
$E_i W \neq 0$ if and only if $t \leq i \leq t+d$;
\item[\rm (ii)]
$E^*_i W \neq 0$ if and only if $r \leq i \leq r+d$.
\end{itemize}
\end{lemma}

\begin{lemma} {\rm (See \cite[Lemma 5.7]{T:nuc}.) }
\label{lem:shape}   \samepage
\ifDRAFT {\rm lem:shape}. \fi
Let $W$ denote an irreducible $T$-module,
with endpoint $r$,
dual endpoint $t$,
and diameter $d$.
Then for $0 \leq i \leq d$ the following subspaces have the same dimension:
\begin{equation}
 E_{t+i} W, \qquad
 E_{t+d-i} W, \qquad
 E^*_{r+i} W, \qquad
 E^*_{r+d-i} W.
\label{eq:subspaces}
\end{equation}
\end{lemma}

\begin{defi}   \label{def:shape}   \samepage
\ifDRAFT {\rm def:shape}. \fi
We refer to the irreducible $T$-module $W$ from  Lemma \ref{lem:shape}.
For $0 \leq i \leq d$ let $\rho_i$ denote the common dimension of the four subspaces
\eqref{eq:subspaces}.
By construction $\rho_i \neq 0$ and $\rho_i = \rho_{d-i}$.
We call the sequence $\{\rho_i\}_{i=0}^d$ the {\em shape of $W$}.
\end{defi}

\begin{defi}   \label{def:thin}   \samepage
\ifDRAFT {\rm def:thin}. \fi 
Let $W$ denote an irreducible $T$-module,
with diameter $d$ and shape $\{\rho_i\}_{i=0}^d$.
Then  $W$ is said to be {\em thin} whenever $\rho_i = 1$  for $0 \leq i \leq d$.
\end{defi}

\begin{lemma} {\rm (See \cite[Lemma 3.4]{T:subconst1}.)}
\label{lem:Vdecomp}    \samepage
\ifDRAFT{\rm lem:Vdecomp}. \fi
Every $T$-module is a direct sum of irreducible $T$-modules.
\end{lemma}

By Lemma \ref{lem:Vdecomp},  there exist irreducible $T$-modules $\{W_h\}_{h=1}^n$ 
such that
\begin{equation}
V = \sum_{h=1}^n W_h \qquad\qquad \text{(direct sum)}.        \label{eq:Vdecomp}
\end{equation}
Let $W$ denote an irreducible $T$-module.
By \cite[Lemma 3.3]{Cur}, the number of summands in \eqref{eq:Vdecomp} that are isomorphic to $W$
is independent of the choice of $\{W_h\}_{h=1}^n$.
This number is denoted by $\text{mult}(W)$.

\section{The nucleus}
\label{sec:nuc}
\ifDRAFT {\rm sec:nuc}. \fi

We continue to discuss the $Q$-polynomial distance-regular graph $\Gamma$ with vertex set $X$.
Fix $x \in X$ and write $T=T(x)$.
In this section we recall the nucleus of $\Gamma$ with respect to $x$.

\begin{lemma}  {\rm (See \cite[Proposition 6.3]{T:nuc}.) }
\label{lem:disp0}   \samepage
\ifDRAFT {\rm lem:disp0}. \fi
Let $W$ denote an irreducible $T$-module, with endpoint $r$, dual endpoint $t$,
and diameter $d$.
Then
\begin{equation}
 r+ t- D + d \geq 0.           \label{eq:disp}
\end{equation}
Moreover, equality holds in \eqref{eq:disp} if and only if both
$t=r$ and $d=D-2r$.
In this case $W$ is thin.
\end{lemma}

\begin{defi} {\rm (See \cite[Definition 6.6]{T:nuc}.) }
Let $W$ denote an irreducible $T$-module, with endpoint $r$, dual endpoint $t$,
and diameter $d$.
By the {\em displacement of $W$} we mean the integer
\[
   r+t-D+d.
\]
\end{defi}

\begin{defi}   {\rm (See \cite[Definition 6.8]{T:nuc}.) }
\label{def:nuc}   \samepage
\ifDRAFT {\rm def:nuc}. \fi
By the {\em nucleus of $\Gamma$ with respect to $x$} we mean the span
of the irreducible $T$-modules that have displacement $0$.
\end{defi}

\begin{lemma}   {\rm (See \cite[Proposition 6.10]{T:nuc}.) }
\label{lem:disp0a}   \samepage
\ifDRAFT {\rm lem:disp0a}. \fi
Let $W$ denote an irreducible $T$-submodule of the nucleus,
with endpoint $r$, dual endpoint $t$, and diameter $d$.
Then $W$ is thin, and 
\[
0 \leq r \leq D/2, \qquad t=r, \qquad d=D-2r.
\]
\end{lemma}

\begin{defi}  {\rm (See \cite[Definition 7.2]{T:nuc}.) }
\label{def:Ni}   \samepage
\ifDRAFT {\rm def:Ni}. \fi
For $0 \leq i \leq d$, define
\[
{\mathcal N}_i = 
(E^*_0 V + E^*_1 V + \cdots + E^*_i V) \cap
(E_0 V + E_1 V + \cdots E_{D-i} V).
\]
\end{defi}

\begin{note}
We have
${\mathcal N}_0 = E^*_0 V$ and
${\mathcal N}_D = E_0 V$.
\end{note}

\begin{lemma}  {\rm (See \cite[Lemma 7.1]{T:nuc}.) }
\label{lem:Ndirect}   \samepage
\ifDRAFT {\rm lem:Ndirect}. \fi
The following sum is direct:
\[
 {\mathcal N}_0 + {\mathcal N}_1 + \cdots + {\mathcal N}_D.
\]
\end{lemma}

\begin{defi}  {\rm (See \cite[Definition 7.4]{T:nuc}.) }
\label{def:N}   \samepage
\ifDRAFT {\rm def:N}. \fi
Define a subspace $\mathcal N$ by
 \[
{\mathcal N} =
 {\mathcal N}_0 + {\mathcal N}_1 + \cdots + {\mathcal N}_D.
\]
\end{defi}

\begin{lemma} {\rm (See \cite[Theorem 7.5]{T:nuc}.) }
\label{lem:nuc}   \samepage
\ifDRAFT {\rm lem:nuc}. \fi
The following are the same:
\begin{itemize}
\item[\rm (i)]
the subspace $\N=\N(x)$ from Definition \ref{def:N};
\item[\rm (ii)]
The nucleus of $\Gamma$ with respect to $x$.
\end{itemize}
\end{lemma}

\section{The Johnson graph $J(N, D)$}
\label{sec:Johnson}
\ifDRAFT {\rm sec:Johnson}. \fi

In this section we recall the Johnson graph.
Let $N$, $D$ denote integers such that $0 \leq D \leq N$.
Let $\Omega$ denote a set of size $N$.
The graph $J(N,D)$ has vertex set $\binom{\Omega}{D}$, and 
two vertices  $x,y$  are adjacent whenever
$|x \cap y| = D-1$.
This graph $J(N, D)$ is  called a {\em Johnson graph}.
The graph $J(N, D)$ is isomorphic to $J(N, N-D)$.
For this reason and to avoid degenerate situations, we assume that
\[
  N > 2D.
\]
By \cite[Section 2.10.3]{BBIT}, the graph $J(N, D)$ is
distance-regular with diameter $D$.
For vertices $x,y$ of $J(N,D)$ we have
\[
 \partial (x,y) = D- |x \cap y|.
\]
By \cite[Theorem 9.1.2]{BCN}, the intersection numbers of $J(N,D)$ satisfy
\begin{align*}
c_i &= i^2  \qquad\qquad\qquad\qquad\qquad (1 \leq i \leq D),
\\
b_i &= (D-i)(N-D-i)  \qquad\;  (0 \leq i \leq D-1).
\end{align*}
By \cite[Proposition 2.87]{BBIT}, 
\begin{equation}
k_i = \begin{pmatrix} D  \\ i \end{pmatrix}
       \begin{pmatrix} N-D \\ i \end{pmatrix}
 \qquad\qquad (0 \leq i \leq D).         \label{eq:ki}
\end{equation}
Consider the adjacency matrix  $A$ of $J(N,D)$.
By \cite[Theorem 2.95]{BBIT},
the eigenvalues of  $A$ are
\begin{equation}
\th_i = (D-i)(N-D-i) - i \qquad\qquad (0 \leq i \leq D).   \label{eq:thi}
\end{equation}
For $0 \leq i \leq D$,
let $E_i$ denote the primitive idempotent of $A$ associated with $\th_i$.
By \cite[Corollary 8.4.2]{BCN}, $J(N,D)$ is $Q$-polynomial with respect to
the ordering $\{E_i\}_{i=0}^D$.
Fix a vertex $x$ of $J(N,D)$, and
abbreviate $T=T(x)$.
By \cite[Theorem 2.95]{BBIT},
the eigenvalues of $A^* = A^*(x)$ are as follows.
If $D=0$ then $\th^*_0 = 0$.
If $D \geq 1$ then
\begin{equation}
\th^*_i = N - 1 -
      \frac{ i N (N-1)} { D(N-D) }      
 \qquad\qquad (0 \leq i \leq D).   \label{eq:thsi}
\end{equation}

\section{The subgraph $\Gamma_i(x)$}
\label{sec:Gammai}
\ifDRAFT {\rm sec:Gamma}. \fi

We continue to discuss the Johnson graph $ J(N,D)$ with the fixed vertex $x$.
In this section we investigate the structure of the subgraph $\Gamma_i(x)$
for $0 \leq i \leq D$.
We construct a graph isomorphism 
\begin{equation}
\Gamma_i(x)  \to  J(D, i) \times J(N-D,i).      \label{eq:iso0}
\end{equation}
This isomorphism shows that the subgraph $\Gamma_i(x)$ has two types of edges,
some come from $J(D,i)$ and some come from $J(N-D,i)$.
We consider a new graph  obtained from
$\Gamma_i(x)$ by removing the edges that come from $J(D,i)$.
We describe the connected components of the new graph.

The fixed vertex $x$ is a subset of $\Omega$ with $|x|=D$.
For the rest of this section, fix an integer $i$  with $0 \leq i \leq  D$.
Consider the Johnson graph $J(D,i)$ with vertex set $\binom{x}{i}$,
and the Johnson graph $J(N-D,i)$ with vertex set
$\binom{\Omega \setminus x}{i}$.
The Cartesian product $J(D,i) \times J(N-D,i)$ has vertex set
\[
  \binom{x}{i} \times \binom{\Omega \setminus x}{i}.
\]
For $y \in \Gamma_i(x)$ we have
\[
  x \setminus y  \in \binom{x}{i},
\qquad\qquad
  y \setminus x  \in \binom{\Omega \setminus x}{i}.
\]
We consider a  map
\begin{equation}
\Gamma_i(x) \to J(D,i) \times J(N-D, i)    \label{eq:natural}
\end{equation}
that sends
\[
y \mapsto (x \setminus y , \, y \setminus x),
\qquad\qquad 
 y \in \Gamma_i(x).
\]
Our next goal is to show that the map \eqref{eq:natural} is a graph isomorphism.

\begin{lemma}   \label{lem:Gaiadjacent}   \samepage
\ifDRAFT {\rm lem:Gaiadjacent}. \fi
Let $y,z$ denote vertices in $\Gamma_i(x)$.
Then $y,z$ are adjacent in the subgraph $\Gamma_i(x)$ if and only if one of the following holds:
\begin{itemize}
\item[\rm (i)]
$x \setminus y =  x \setminus z$, and $\;\;y \setminus  x$, $z\setminus x$ are adjacent
in $J(N-D,i)$;
\item[\rm (ii)]
$y \setminus x = z \setminus x$,  and $x \setminus y$, $x \setminus z$ are adjacent
 in $J(D,i)$.
\end{itemize}
\end{lemma}

\begin{proof}
Consider the Venn diagram for $x$, $y$, $z$:
\newpage
\begin{center}
\begin{picture}(200,100)
\put(0,0){\circle{100}}
\put(50,0){\circle{100}}
\put(25,50){\circle{100}}
\put(25,105){$x$}
\put(-60,0){$y$}
\put (105,0){$z$}
\put(25, 60){$a$}
\put(-20,-15){$b$}
\put(70,-15){$c$}
\put(25,-20){$d$}
\put(-10,30){$e$}
\put(55,30){$f$}
\put(25,20){$g$}
\end{picture}
\end{center}
\vspace{2cm}
Since $|x|=|y|=|z| = D$,
\begin{align}
& |a|+|e|+|f|+|g|=D    \label{eq:aux1},
\\
& |b|+|d|+|e|+|g|=D,    \label{eq:aux2}
\\
&|c|+|d|+|f|+|g|=D.      \label{eq:aux3}
\end{align}
Since $y, z \in \Gamma_i(x)$, we have
$|x \cap y| = D-i$ and $|x \cap z| = D-i$. So
\begin{align}
& |e|+|g| = D-i, 
\qquad\qquad 
|f|+|g| = D-i.    \label{eq:aux4}
\end{align}
The vertices $y,z$ are adjacent if and only if
\begin{equation}
  |d| + |g| = D-1.              \label{eq:cond1}
\end{equation}
Condition (i) holds if and only if
\begin{equation}
 |e| = 0, \qquad
 |f| = 0 , \qquad
 |d| = i-1.                \label{eq:cond2}
\end{equation}
Condition (ii) holds if and only if
\begin{equation}
 |b| = 0, \qquad
 |c| = 0 , \qquad
 |a| = i-1.                \label{eq:cond3}
\end{equation}
Using \eqref{eq:aux1}--\eqref{eq:aux4} we  routinely
find that \eqref{eq:cond1} holds if and only if
one of \eqref{eq:cond2}, \eqref{eq:cond3} holds.
\end{proof}

\begin{lemma}  \label{lem:natural}   \samepage
\ifDRAFT {\rm lem:natural}. \fi
The map \eqref{eq:natural}
is a graph isomorphism.
\end{lemma}

\begin{proof}
We first show that  the map \eqref{eq:natural} is a bijection. 
Clearly, the map \eqref{eq:natural} is injective.
By \eqref{eq:ki} the graphs $\Gamma_i(x)$ and 
$J(D,i) \times J(N-D,i)$ have the same number of vertices.
By these comments, the map \eqref{eq:natural} is a bijection.
The map \eqref{eq:natural} respects adjacency by Lemma \ref{lem:Gaiadjacent}.
The result follows.
\end{proof}

Referring to the isomorphism \eqref{eq:natural},
consider a new graph that is obtained from the subgraph $\Gamma_i (x)$
by removing the edges that come from $J(D,i)$.
Our next goal is to describe the connected components of the new graph.
By Lemma \ref{lem:cartesian}, these connected components are
the preimages under \eqref{eq:natural} of the sets
\[
    \{\alpha\} \times \binom{\Omega \setminus x}{i},
  \qquad \qquad  \alpha \in \binom{x}{i}.
\]

\begin{lemma}   \label{lem:Calpha0}   \samepage
\ifDRAFT {\rm lem:Calpha0}. \fi
Referring to the isomorphism \eqref{eq:natural},
for $\alpha  \in  \binom{x}{i}$,
the preimage of  $ \{\alpha\} \times \binom{\Omega \setminus x}{i}$
is the set
\begin{align}
& \{y \in X \;|\; x \setminus y =  \alpha\}.     \label{eq:C+0} 
\end{align}
\end{lemma}

\begin{proof}
Clear by construction.
\end{proof}

\begin{lemma}   \label{lem:Calpha2}   \samepage
\ifDRAFT {\rm lem:Calpha2}. \fi
Consider a new graph obtained from the subgraph  $\Gamma_i(x)$
by removing the edges that come from $J(D,i)$.
For the new graph 
the connected components are
\begin{align}
& \{ y \in X \;|\; y \cap x = \alpha\},
\qquad\qquad
 \alpha \subseteq x, \;\; |\alpha| =D- i.        \label{eq:C+}
\end{align}
\end{lemma}

\begin{proof}
By the claim above Lemma \ref{lem:Calpha0} and Lemmas  \ref{lem:natural} and \ref{lem:Calpha0}.
\end{proof}

\begin{corollary}    \label{cor:newgraph}   \samepage
\ifDRAFT {\rm cor:newgraph}. \fi
Referring to Lemma \ref{lem:Calpha2},
each connected component of the new graph is isomorphic to $J(N-D, i)$.
For the new graph the number of the  connected components  is 
{\scriptsize $\begin{pmatrix} D  \\  i  \end{pmatrix}$}.
\end{corollary}

\begin{proof}
By Lemmas \ref{lem:Calpha0} and \ref{lem:Calpha2}.
\end{proof}

\section{The vectors $\alpha^\vee$ and $\alpha^\N$}
\label{sec:alphaN}
\ifDRAFT {\rm sec:alphaN}. \fi

We continue to discuss the Johnson graph $ J(N,D)$ with the fixed vertex $x$.
Recall the standard module $V = \C^X$.
We introduce some vectors in $V$ that will form a basis
of the nucleus with respect to $x$.

\begin{defi}  \label{def:alphaN}   \samepage
\ifDRAFT {\rm def:alphaN}. \fi
For a subset $\alpha \subseteq x$, define vectors in $V$:
\[
\alpha^\vee =
\sum_{\text{\scriptsize $\begin{matrix} y \in X \\ \alpha \subseteq y \end{matrix}$}} \widehat{y},
\qquad\qquad\qquad
\alpha^\N =
\sum_{\text{\scriptsize $\begin{matrix} y \in X \\ y \cap x = \alpha \end{matrix}$}} \widehat{y}.
\]
The vector $\alpha^\vee$ is the characteristic vector
of the set of vertices in $X$ that contain $\alpha$,
and  $\alpha^\N$ is  the characteristic vector
of the set of vertices in $X$ whose intersection with $x$ is equal to $\alpha$.
\end{defi}

The vectors $\alpha^\vee$ and $\alpha^\N$ are related as follows.

\begin{lemma} \label{lem:Nvee1} \samepage
\ifDRAFT {\rm lem:Nvee1}. \fi
For $\alpha \subseteq x$ the following hold:
\begin{itemize}
\item[\rm (i)]
$ \displaystyle
\alpha^\vee
= \sum_{\alpha \subseteq \beta \subseteq x} \beta^\N.
$
\item[\rm (ii)]
$\displaystyle  \alpha^N = \sum_{\alpha \subseteq \beta \subseteq x}
           (-1)^{|\beta|-|\alpha|} \,  \beta^\vee.
$
\end{itemize}
\end{lemma}

\begin{proof}
(i)
We argue
\begin{align*}
\alpha^\vee &=
    \sum_{\text{\scriptsize $\begin{matrix} y \in X \\ \alpha \subseteq y \end{matrix}$}} \widehat{y}
=
\sum_{\alpha \subseteq \beta \subseteq x}
 \sum_{\text{\scriptsize $\begin{matrix} y \in X \\ y \cap x = \beta \end{matrix}$}} \widehat{y}
=
\sum_{\alpha \subseteq \beta \subseteq x} \beta^\N.
\end{align*}

(ii)
Using (i) we obtain
\begin{align*}
\sum_{\alpha \subseteq \beta \subseteq x}
           (-1)^{|\beta|-|\alpha|} \,  \beta^\vee
&= \sum_{\alpha \subseteq \beta \subseteq x } (-1)^{|\beta| - | \alpha|} \sum_{\beta \subseteq \gamma \subseteq x}
         \gamma^\N
\\
& = \sum_{\alpha \subseteq \gamma \subseteq x}
  \gamma^\N \sum_{\alpha \subseteq \beta \subseteq \gamma} (-1)^{|\beta| - |\alpha|}.
\end{align*}
For $\alpha \subseteq \gamma \subseteq x$ we have
\begin{align*}
\sum_{\alpha \subseteq \beta \subseteq \gamma} (-1)^{|\beta| - |\alpha|}
&=
\sum_{j=0}^{|\gamma| - |\alpha|} \binom{|\gamma| - |\alpha|}{j} (-1)^j
\\
&= (1-1)^{|\gamma| - |\alpha|}
\\
 &=
\begin{cases}
  1 & \text{ if $\gamma = \alpha$},
   \\
  0 & \text{ if $\gamma \neq \alpha$}.
\end{cases}
\end{align*}
The result follows.
\end{proof}

\begin{lemma}   \label{lem:Nvee2}   \samepage
\ifDRAFT {\rm lem:Nvee2}. \fi
For $\alpha \subseteq x$, 
\[
    \alpha^N = E^*_i \alpha^\vee,
\]
where $i = D-|\alpha|$.
\end{lemma}

\begin{proof}
For $y \in X$ and $0 \leq j \leq D$,
we have  $j=D-|x \cap y|$ if and only if $y \in \Gamma_j(x)$.
The result follows from this and  Definition \ref{def:alphaN}.
\end{proof}

By  Lemma \ref{lem:Calpha2},
\begin{equation}
\{\alpha^\N \;|\;\alpha \subseteq x \} \quad
\text{ are linearly independent}.        \label{eq:claim2(i)}
\end{equation}
By Lemma \ref{lem:Nvee2} and  \eqref{eq:claim2(i)},
\begin{equation}
\{\alpha^\vee \;|\;\alpha \subseteq x \} \quad
\text{ are linearly independent}.    \label{eq:claim2(ii)}
\end{equation}

\section{The action of $A$, $A^*$ on the vectors 
$\alpha^\vee$, 
$\alpha^\N$}
\label{sec:action}
\ifDRAFT {\rm sec:actiion}. \fi

We continue to discuss the Johnson graph $J(N,D)$ with the fixed vertex $x$.
Recall the vectors $\alpha^\vee$, $\alpha^\N$  from Definition \ref{def:alphaN}.
In this section we display the action of $A$, $A^*$ on these vectors.

\begin{theorem}   \label{thm:Aalphavee}   \samepage
\ifDRAFT {\rm thm:Aalphavee}. \fi
For  $\alpha \subseteq x$, 
\begin{align}
A \alpha^\vee
&=
 \th_{|\alpha|}  \alpha^\vee
 + (D-|\alpha|+1)
 \sum_
    {\text{ \scriptsize $\begin{matrix}  \beta \subseteq \alpha  \\  |\beta| = |\alpha|-1 \end{matrix}$}}
       \beta^\vee.             \label{eq:Aalphavee}
\end{align}
\end{theorem}

\begin{proof}
Let the set $S_0$ consist of the vertices in $X$ that include $\alpha$.
The characteristic vector of $S_0$ is equal to $\alpha^\vee$.
Let the set $S_1$ consist of the vertices in $X \setminus S_0$
that are adjacent to at least one vertex in $S_0$.
The characteristic vector of $S_1$ is equal to
\[
 \sum_
    {\text{ \scriptsize $\begin{matrix}  \beta \subseteq \alpha  \\  |\beta| = |\alpha|-1 \end{matrix}$}}
    (\beta^\vee - \alpha^\vee).
\]
By combinatorial counting, we find that
each vertex in $S_0$ is adjacent to exactly
$(D-|\alpha|)(N-D)$ vertices in $S_0$.
Also, each vertex in $S_1$ is adjacent to exactly $D-|\alpha|+1$ vertices in $S_0$.
By these comments,
\[
A \alpha^\vee =
 (D-|\alpha|)(N-D) \alpha^\vee
+ (D-|\alpha|+1) 
 \sum_
    {\text{ \scriptsize $\begin{matrix}  \beta \subseteq \alpha  \\  |\beta| = |\alpha|-1 \end{matrix}$}}
      ( \beta^\vee - \alpha^\vee).
\]
By this and \eqref{eq:thi} we obtain  \eqref{eq:Aalphavee}.
\end{proof}

\begin{theorem}   \label{thm:AalphaN}   \samepage
\ifDRAFT {\rm thmAalphaN}. \fi
For $\alpha \subseteq x$,
\begin{align*}
A \alpha^\N 
&=
(D-|\alpha|) (N-2D+|\alpha|) \alpha^\N
\;\; +\;\;  (N-2D+|\alpha|+1 ) 
\sum_
  {\text{\scriptsize 
     $\begin{matrix} \alpha \subseteq \beta \subseteq x \\ |\beta| = | \alpha|+1 \end{matrix}$
  }}
    \beta^\N
\\ & \qquad
+ \sum_{\text{\scriptsize $\begin{matrix} \beta \subseteq x \\ |\beta| = |\alpha|  \\ |\beta \cap \alpha| = |\alpha| -1 \end{matrix}$ }}
     \beta^\N
\;\; + \;\;  (D-|\alpha|+1)
\sum_{\text{\scriptsize $\begin{matrix}  \beta \subseteq \alpha \\ |\beta| = |\alpha| - 1 \end{matrix}$}}
  \beta^N.
\end{align*}
\end{theorem}

\begin{proof}
The vector  $\alpha^\N$ is the characteristic vector of the set
\begin{equation}
 \{ y \in X \;|\; y \cap x = \alpha\}.    \label{eq:Lamalpha}
\end{equation}
So we have
\[
 A \alpha^N = \sum_{z \in X}  e_z \widehat{z},
\]
where $e_z$ is the number of edges from $z$ into the set \eqref{eq:Lamalpha}.
We now compute $e_z$ for $z \in X$.
Let $z$ be given and define $\beta = z \cap x$.
By combinatorial counting, we find that $e_z=0$
or
$e_z$ is given in the table below.
\[
\begin{array}{c|c}
\text{Case}  & e_z
\\ \hline
\beta=\alpha & (D-|\alpha|)(N-2D+|\alpha|)      \rule{0mm}{3.2ex}
\\
\text{$\alpha \subseteq \beta$ and $|\beta| = |\alpha| + 1$} & N-2D + |\alpha| + 1   \rule{0mm}{3ex}
\\
\text{$|\beta| = |\alpha|$ and $|\beta \cap \alpha| = |\alpha| -1$} & 1                  \rule{0mm}{3ex}
\\
\text{$\beta \subseteq \alpha$ and $|\beta| = |\alpha| -1$}  & D-|\alpha| + 1       \rule{0mm}{3ex}
\end{array}
\]
Using the data in the table, we routinely obtain the result.
\end{proof}

\begin{theorem}  \label{thm:AsalphaN}   \samepage
\ifDRAFT {\rm thm:AsalphaN}. \fi
For $\alpha \subseteq x$, 
\[
A^* \alpha^\N = \th^*_{D-|\alpha|} \alpha^\N.
\]
\end{theorem}

\begin{proof}
Clear since $\alpha^\N \in E^*_i V$, where $i=D-|\alpha|$.
\end{proof}

\begin{theorem}    \label{thm:Asalphsvee}   \samepage
\ifDRAFT {\rm thm:Asalphavee}. \fi
For $\alpha \subseteq x$,
\[
A^* \alpha^\vee =
  \th^*_{D-|\alpha|}  \alpha^\vee
 + \frac{N(N-1)}{D(N-D)} 
 \sum_{\text{\scriptsize $\begin{matrix} \alpha \subseteq \gamma \subseteq x \\ |\gamma| = |\alpha| + 1 \end{matrix}$}}
    \gamma^\vee.
\]
\end{theorem}

\begin{proof}
Using  Lemma \ref{lem:Nvee1}(i), Theorem \ref{thm:AsalphaN}, and Lemma \ref{lem:Nvee1}(ii) in order,
\begin{align*}
A^* \alpha^\vee &=
  \sum_{\alpha \subseteq \beta \subseteq x} A^* \beta^\N.
\\
&= \sum_{\alpha \subseteq \beta \subseteq x} \th^*_{D-|\beta|} \beta^\N
\\
&= \sum_{\alpha \subseteq \beta \subseteq x} \th^*_{D-|\beta|} 
    \sum_{\beta \subseteq \gamma \subseteq x} (-1)^{|\gamma| - |\beta|} \gamma^\vee.
\end{align*}
Thus
\begin{equation}
A^* \alpha^\vee 
= \sum_{\alpha \subseteq \gamma \subseteq x}
   \gamma^\vee
    \sum_{\alpha \subseteq \beta \subseteq \gamma}
     (-1)^{|\gamma| - |\beta|} \th^*_{D-|\beta|}.
       \label{eq:rhs1}
\end{equation}
For $\alpha \subseteq \beta \subseteq \gamma$ we have
\begin{align*}
\sum_{\alpha \subseteq \beta \subseteq \gamma} (-1)^{|\gamma| - |\beta|} \th^*_{D-|\beta|}
&=
\sum_{i=|\alpha|}^{|\gamma|}  (-1)^{|\gamma| - i} \binom{|\gamma| - {\alpha}}{ i - |\alpha|} \th^*_{D-i}
\\
&=
\sum_{j=0}^{|\gamma|- |\alpha|}  (-1)^{|\gamma| - |\alpha|-j} \binom{  |\gamma|-|\alpha| } {j} \th^*_{D-|\alpha|-j}.
\end{align*}
Evaluating this using  \eqref{eq:thsi} we obtain
\[
\sum_{\alpha \subseteq \beta \subseteq \gamma} (-1)^{|\gamma| - |\beta|} \th^*_{D-|\beta|}
 = 
 \begin{cases}
    0 & \text{ if $|\gamma|- |\alpha| \geq 2$},
\\
 \frac{N(N-1)}{D(N-D)}  & \text{ if $|\gamma| - |\alpha| =1$},
\\
  \th^*_{D-|\alpha|} & \text{ if $\gamma =\alpha$}
\end{cases}
\]
By this and \eqref{eq:rhs1} we get the result.
\end{proof}

\section{Some results concerning the vectors $\alpha^\vee$ and $\alpha^\N$}
\label{sec:alphavee}
\ifDRAFT {\rm sec:alphavee}. \fi

We continue to discuss the Johnson graph $J(N, D)$ with the fixed vertex $x$.
Recall the vector $\alpha^\vee$ and $\alpha^\N$  from Definition \ref{def:alphaN}.
In this section we show that the vector $\alpha^\vee$ and $\alpha^\N$ are contained in the nucleus
of $J(N,D)$ with respect to $x$.

The following result is known \cite{Tanaka}.
We give a short direct proof for the sake of completeness.

\begin{lemma}   {\rm (See \cite[Example (7.1)]{Tanaka}.) }
 \label{lem:MiErV}   \samepage
\ifDRAFT {\rm lem:MiErV}. \fi
For $\alpha\subseteq x$, 
\[
 \alpha^\vee \in E_0 V + E_1 V + \cdots + E_{|\alpha|} V.
\]
\end{lemma}

\begin{proof}
We may assume that $\alpha \neq x$;
otherwise the result is trivial.
We have $|\alpha| < D$.
By Theorem  \ref{thm:Aalphavee}, for $\gamma \subseteq x$ we have
\[
  (A-\th_{|\gamma|} I)\, \gamma^\vee
  \in \text{\rm Span} \{\beta^\vee \;|\;  \beta \subseteq \gamma, \; |\beta|=|\gamma|-1 \}.
\]
Using this and induction on $|\alpha|$,  we obtain
\[
\tau_{|\alpha|+1} (A) \alpha^\vee = 0.
\]
By this and \eqref{eq:Er},
\[
  E_r \alpha^\vee = 0
   \qquad\qquad (|\alpha| < r \leq D).
\]
The result follows.
\end{proof}

\begin{lemma}   \label{lem:MiEsrV}   \samepage
\ifDRAFT {\rm lem:MiEsrV}. \fi
For $\alpha\subseteq x$, 
\[
 \alpha^\vee \in
  E^*_0 V + E^*_1 V + \cdots + E^*_{D-|\alpha|} V.
\]
\end{lemma}

\begin{proof}
For $y \in X$ we have $\widehat{y} \in E^*_i V$, where $i=D-|y \cap x|$.
By Definition \ref{def:alphaN},
\[
\alpha^\vee = 
\sum_{
    \text{  \scriptsize 
                  $\begin{matrix} y \in X \\ \alpha \subseteq y \end{matrix}$
   }} 
\widehat{y}.
\]
By these comments we get the result.
\end{proof}

\begin{lemma}   \label{lem:claim1}   \samepage
\ifDRAFT {\rm lem:claim1}. \fi
For  $\alpha \subseteq x$,
\begin{equation}
\alpha^\vee \in \N_{D-|\alpha|}.    \label{eq:claim1}
\end{equation}
\end{lemma}

\begin{proof}
By Definition \ref{def:Ni} and Lemmas \ref{lem:MiErV}, \ref{lem:MiEsrV}.
\end{proof}

\begin{lemma}      \label{lem:claim1b}   \samepage
\ifDRAFT {\em lem:claim1b}.  \fi
For $\alpha \subseteq x$,
\[
   \alpha^\N \in E^*_{D-|\alpha|} \N.
\]
\end{lemma}

\begin{proof}
By Lemmas \ref{lem:Nvee2} and \ref{lem:claim1}.
\end{proof}

\section{Some bases for the nucleus}
\label{sec:main}
\ifDRAFT {\rm sec:main}. \fi

We continue to discuss the Johnson graph $J(N, D)$ with the fixed vertex $x$.
In this section we obtain some bases for the nucleus of with respect to $x$.
Let $W$ denote an irreducible $T$-module,
with endpoint $r$, dual endpoint $s$, and diameter $d$.
By \cite[Section 2.3]{IL}, the isomorphism class of $W$ is determined by 
$r$, $t$, $d$.
Assume that $W$ has displacement $0$.
Then by Lemma \ref{lem:disp0}, $r=t$ and $0 \leq r \leq D/2$ and $d=D-2r$.
In this case the isomorphism class of $W$ is determined by $r$.
Abbreviate
\[
\text{mult}_r = \text{mult} (W).
\]

\begin{lemma}   {\rm (See \cite[Theorem 10.2]{Go}, \cite[Remark 2]{IL}.) }
\label{lem:multr}   \samepage
\ifDRAFT {\rm lem:multr}. \fi
We have $\text{\rm mult}_0 = 1$ and
\[
\text{\rm mult}_r
 = 
  \begin{pmatrix} D \\ r  \end{pmatrix}
   -  \begin{pmatrix} D \\ r - 1 \end{pmatrix}
  \qquad\qquad (1 \leq r \leq D/2).
\]
\end{lemma}

\begin{lemma}   \label{lem:dimEsiN}   \samepage
\ifDRAFT {\rm lem:dimEsiN}. \fi
For $0 \leq i  \leq D$,
\begin{equation}
\text{\rm dim} E^*_i \N =
 \begin{pmatrix} D \\ i \end{pmatrix}.           \label{eq:dimEsiN}
\end{equation}
\end{lemma}

\begin{proof}
By Lemmas \ref{lem:Vdecomp} and \ref{lem:disp0a},  
there exist irreducible $T$-modules $\{W_h\}_{h=1}^m$ with displacement $0$ such that
\[
 \N = \sum_{h=1}^m W_h
\qquad\qquad  \text{(direct sum)}.
\]
So
\[
  \dim E^*_i \N = \sum_{h=1}^m \dim E^*_i W_h.
\]
For $1 \leq h \leq m$ we have  $\dim E^*_i W_h = \dim E^*_{D-i} W_h$ by Lemmas \ref{lem:shape} and \ref{lem:disp0a}.
So $\dim E^*_i \N = \dim E^*_{D-i} \N$.
By this  we may assume $0 \leq i \leq D/2$.
For $1 \leq h \leq m$ let $r_h$ denote the endpoint of $W_h$.
By Lemma \ref{lem:disp0} and the definition of the endpoint, we have
\[
 \dim E^*_i W_h = 
 \begin{cases}
    0   &  \text{ if $i < r_h$},
\\
   1    &   \text{ if $r_h \leq i$}
 \end{cases}
\qquad\qquad (1 \leq h \leq m).
\]
For $0 \leq r \leq i$, the number of $h$ $(1 \leq h \leq m)$ such that
$r_h = r$ is equal to $\text{mult}_r$.
By these comments,
\[
  \dim E^*_i \N = \sum_{r=0}^i \text{mult}_r.
\]
Evaluating this sum using  Lemma \ref{lem:multr},
we routinely obtain \eqref{eq:dimEsiN}.
\end{proof}

\begin{lemma}   \label{lem:claim3}   \samepage
\ifDRAFT {\rm lem:claim3}. \fi
We have
\begin{itemize}
\item[\rm (i)]
$\dim \N = 2^D$;
\item[\rm (ii)]
$\dim \N_i = \text{\scriptsize $\begin{pmatrix} D  \\  i \end{pmatrix}$}$  for $0 \leq i \leq D$.
\end{itemize}
\end{lemma}

\begin{proof}
 (i)
We have $\N = \sum_{i=0}^D E^*_i \N$ (direct sum).
By this and Lemma \ref{lem:dimEsiN}, we get the result.

(ii)
By Lemma \ref{lem:Ndirect}, 
$\N = \sum_{i=0}^D \N_i$ (direct sum).
By this and (i),
\[
2^D = \dim \N = \sum_{i=0}^D \dim \N_i.
\]
By \eqref{eq:claim2(ii)} and \eqref{eq:claim1},
\[
  \dim \N_i \geq \begin{pmatrix} D \\ i \end{pmatrix}
 \qquad\qquad (0 \leq i \leq D).
\]
By these comments we get the result.
\end{proof}

\begin{theorem}   \label{thm:main}   \samepage
\ifDRAFT {\rm thm:main}. \fi
The following hold:
\begin{itemize}
\item[\rm (i)]
The vectors
$\{\alpha^\N\}_{\alpha \subseteq x}$ form a basis of $\N$.
\item[\rm (ii)]
The vectors
$\{\alpha^\vee\}_{\alpha \subseteq x}$ form a basis of $\N$.
\item[\rm (iii)]
For $0 \leq i \leq D$,
the vectors
$\{\alpha^\N \;|\;  \alpha \subseteq x, \; |\alpha|=D-i\}$
form a basis of $E^*_i {\mathcal N}$.
\item[\rm (iv)]
For $0 \leq i \leq D$,
the vectors
$\{\alpha^\vee \;|\;  \alpha  \subseteq x, \; |\alpha|=D-i\}$
form a basis of ${\mathcal N}_i$.
\end{itemize}
\end{theorem}

\begin{proof}
(i)
By Lemma \ref{lem:claim1b}, $\alpha^\N \in \N$.
Now the result follows from  \eqref{eq:claim2(i)} and Lemma \ref{lem:claim3}(i).

(ii)
By Lemma \ref{lem:claim1}, $\alpha^\vee \in \N$.
Now the result follows from \eqref{eq:claim2(ii)} and Lemma \ref{lem:claim3}(i).

(iii)
By Lemma \ref{lem:claim1b}, $\alpha^\N \in E^*_i \N$ if $|\alpha| = D-i$.
Now the result follows from
Lemma \ref{lem:dimEsiN} and \eqref{eq:claim2(i)}.

(iv)
By Lemma \ref{lem:claim1}, $\alpha^\vee \in \N_i$ if $|\alpha| = D-i$.
Now the result follows from
\eqref{eq:claim2(ii)} and Lemma \ref{lem:claim3}(ii).
\end{proof}

Recall by Lemma \ref{lem:natural} that
for $0 \leq i \leq D$ the subgraph $\Gamma_i(x)$ is isomorphic to
$J(D,i) \times J(N-D,i)$.

\begin{corollary}   \label{cor:main}   \samepage
\ifDRAFT {\rm cor:main}. \fi
For $0 \leq i \leq D$,
consider a new graph obtained from the subgraph $\Gamma_i(x)$ by
removing the edges that come from $J(D,i)$.
Then the vectors in Theorem \ref{thm:main}{\rm (iii)} are the characteristic vectors 
of the connected components of the new graph.
\end{corollary}

\begin{proof}
By Lemma \ref{lem:Calpha2} the sets
\[
  \{y \in X \;|\; y \cap x =\alpha\},
  \qquad\qquad \alpha\subseteq x, \; |\alpha| = D-i
\]
are the connected components of the new graph.
Recall  that $\alpha^\N$ is the characteristic vector of
$\{y \in X \;|\; y \cap x = \alpha\}$.
By these comments, we get the result.
\end{proof}

{

}

\bigskip\bigskip\noindent
Kazumasa Nomura\\
Institute of Science Tokyo\\
Kohnodai, Ichikawa, 272-0827 Japan\\
email: knomura@pop11.odn.ne.jp

\bigskip\noindent
Paul Terwilliger\\
Department of Mathematics\\
University of Wisconsin\\
480 Lincoln Drive\\ 
Madison, Wisconsin, 53706 USA\\
email: terwilli@math.wisc.edu

\bigskip\noindent
{\bf Keywords.}
Johnson graph,
distance-regular graph,
$Q$-polynomial.
nucleus.
\\
\noindent
{\bf 2020 Mathematics Subject Classification.} 
05E30,  15B10.

\end{document}